\documentclass[a4paper,12pt,reqno]{amsart}
\usepackage{a4wide}
\usepackage{amsfonts,amsmath,amssymb}
\usepackage{graphicx}
\usepackage{tikz}

\newtheorem{theo}{Theorem}

\newtheorem{lemma}[theo]{Lemma}
\newtheorem{cor}[theo]{Corollary}

\theoremstyle{remark}

\newtheorem{question}{Question}

\begin{document}

\thispagestyle{plain}

\title{Paths vs. stars in the local profile of trees}

\author{\'Eva Czabarka, L\'aszl\'o A. Sz\'ekely }
\address{\'Eva Czabarka and L\'aszl\'o A. Sz\'ekely\\ Department of Mathematics \\ University of South Carolina \\ Columbia, SC 29208 \\ USA}
\email{\{czabarka,szekely\}@math.sc.edu }
\thanks{The second author was supported in part by the  NSF DMS,  grant number 1300547,  the third author was supported in part by the National
Research Foundation of South Africa, grant number 96236.}
\author{Stephan Wagner}
\address{Stephan Wagner\\ Department of Mathematical Sciences \\ Stellenbosch University \\ Private Bag X1, Matieland 7602 \\ South Africa}
\email{swagner@sun.ac.za}
\subjclass[2010]{05C05}
\keywords{trees, subtrees, local profile, paths, stars}

\date{\today}

\begin{abstract}
The aim of this paper is to provide an affirmative answer to a recent question by Bubeck and Linial on the local profile of trees. For a tree $T$, let $p^{(k)}_1(T)$ be the proportion of paths among all $k$-vertex subtrees (induced connected subgraphs) of $T$, and let $p^{(k)}_2(T)$ be the proportion of stars. Our main theorem states: if $p^{(k)}_1(T_n) \to 0$ for a sequence of trees $T_1,T_2,\ldots$ whose size tends to infinity, then $p^{(k)}_2(T_n) \to 1$. Both are also shown to be equivalent to the statement that the number of $k$-vertex subtrees grows superlinearly and the statement that the $(k-1)$th degree moment grows superlinearly.
\end{abstract}

\maketitle

\section{introduction}

In their recent paper \cite{bubeck2015local}, Bubeck and Linial studied what they call the \emph{local profile} of trees. For two trees $S$ and $T$, we denote the number of copies of $S$ in $T$ 
by $c(S,T)$ (formally, the number of vertex subsets of $T$ that induce a tree isomorphic to $S$). For an integer $k \geq 4$, let $T_1^{k},T_2^k,\ldots$ be a list of all $k$-vertex trees (up to isomorphism), such that $T_1^k = P_k$ is the path and $T_2^k = S_k$ is the star, and set
$$p^{(k)}_i(T) = \frac{c(T_i^k,T)}{Z_k(T)},\quad \text{where}\quad Z_k(T) = \sum_{j} c(T_j^k,T).$$
In words, $Z_k(T)$ is the number of $k$-vertex subtrees of $T$ (the number of $k$-vertex subsets that induce a tree), and $p^{(k)}_i$ the proportion of copies of $T_i^k$ among those subtrees. In particular, $p_1^{(k)}(T)$ is the proportion of paths, and $p_2^{(k)}(T)$ is the proportion of stars. The vector $p^{(k)}(T) = (p^{(k)}_1(T),p^{(k)}_2(T),\ldots)$ is called the $k$-profile of $T$.

Bubeck and Linial study specifically the limit set $\Delta(k)$ of $k$-profiles $p^{(k)}(T)$ as the number of vertices of $T$ tends to infinity. Their main result is that $\Delta(k)$ is convex for every $k$. This contrasts the situation for general graphs, where the analogously defined set is not convex and even determining the convex hull is computationally infeasible \cite{hatami2011undecidability}. Even in special cases, fairly little is known about $k$-profiles (see \cite{huang2014local} for a study of $3$-profiles). We remark that there is also a notable difference in the definitions of $k$-profiles of general graphs and trees: for graphs, the proportion is taken among all vertex sets of cardinality $k$, while for trees it makes more sense to only consider those $k$-vertex sets that actually induce a tree. For general graphs, this would amount to considering only those subsets that induce a connected graph.

Furthermore, Bubeck and Linial show that the sum of the first two components (corresponding to the path and the star respectively) is strictly positive for every point in the limit set $\Delta(k)$ and in fact bounded below by an explicit constant that only depends on $k$ (see the discussion at the end of Section~\ref{sec:main} and in particular Corollary~\ref{cor:lower_bound_mix} for an equivalent statement). They also obtain a somewhat stronger inequality in the special case $k=5$.

Bubeck and Linial list a number of open problems at the end of their article, and one of them will be the main topic of this paper. It can be expressed as follows:

\begin{question}
Let $T_1,T_2,\ldots$ be a sequence of trees such that the number of vertices of $T_n$ tends to infinity as $n \to \infty$. Given that $\lim_{n \to \infty} p_1^{(k)}(T_n) = 0$, is it necessarily true that $\lim_{n \to \infty} p_2^{(k)}(T_n) = 1$?
\end{question}

In somewhat more informal terms, this states the following: if only few of the $k$-vertex subtrees of a large tree are paths, almost all of those subtrees have to be stars. We remark that the statement is not true if $p_1^{(k)}$ and $p_2^{(k)}$ are interchanged. For example, consider the sequence of caterpillars as shown in Figure~\ref{fig:cater}.

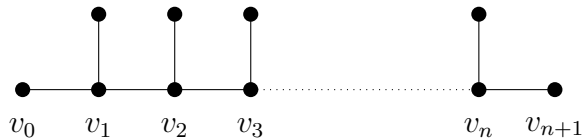
\begin{figure}[htbp]
\begin{center}
\begin{tikzpicture}

        \node[fill=black,circle,inner sep=2pt]  at (0,0) {};
        \node[fill=black,circle,inner sep=2pt]  at (1,0) {};
        \node[fill=black,circle,inner sep=2pt]  at (2,0) {};
        \node[fill=black,circle,inner sep=2pt]  at (3,0) {};
        \node[fill=black,circle,inner sep=2pt]  at (6,0) {};
        \node[fill=black,circle,inner sep=2pt]  at (7,0) {};

        \node[fill=black,circle,inner sep=2pt]  at (1,1) {};
        \node[fill=black,circle,inner sep=2pt]  at (2,1) {};
        \node[fill=black,circle,inner sep=2pt]  at (3,1) {};
        \node[fill=black,circle,inner sep=2pt]  at (6,1) {};

	\draw (0,0)--(3,0);
	\draw [dotted] (3,0)--(6,0);
	\draw (6,0)--(7,0);
	\draw (1,0)--(1,1);                                                                                                                                                                                                                                                                                                                                                                                                                                                                                                                                                                                                                                                                                                                                                                                                                                                                                                                                                                                                                                                                                                                                                                                                                                                                                                                                                                                                                                                                                                                                                                                                                                                                                                                                                                                                                                                                                                                                                                                                                                                                                                                                                      
	\draw (2,0)--(2,1);                                                                                                                                                                                                                                                                                                                                                                                                                                                                                                                                                                                                                                                                                                                                                                                                                                                                                                                                                                                                                                                                                                                                                                                                                                                                                                                                                                                                                                                                                                                                                                                                                                                                                                                                                                                                                                                                                                                                                                                                                                                                                                                                                      
	\draw (3,0)--(3,1);                                                                                                                                                                                                                                                                                                                                                                                                                                                                                                                                                                                                                                                                                                                                                                                                                                                                                                                                                                                                                                                                                                                                                                                                                                                                                                                                                                                                                                                                                                                                                                                                                                                                                                                                                                                                                                                                                                                                                                                                                                                                                                                                                      
	\draw (6,0)--(6,1);                                                                                                                                                                                                                                                                                                                                                                                                                                                                                                                                                                                                                                                                                                                                                                                                                                                                                                                                                                                                                                                                                                                                                                                                                                                                                                                                                                                                                                                                                                                                                                                                                                                                                                                                                                                                                                                                                                                                                                                                                                                                                                                                                      

	\node at (0,-0.5) {$v_0$};
	\node at (1,-0.5) {$v_1$};
	\node at (2,-0.5) {$v_2$};
	\node at (3,-0.5) {$v_3$};
	\node at (6,-0.5) {$v_n$};
	\node at (7,-0.5) {$v_{n+1}$};
\end{tikzpicture}
\end{center}
\caption{A caterpillar.}\label{fig:cater}
\end{figure}

Obviously, $p_2^{(5)}(T_n) = 0$ for every $n$ in this example: the maximum degree is $3$, so $T_n$ does not contain any $5$-vertex stars. On the other hand, simple calculations show that $\lim_{n \to \infty} p_1^{(5)}(T_n) = \frac12$.

In the following, we will provide an affirmative answer to the question raised by Bubeck and Linial, and even prove a slight extension involving the total number of $k$-vertex subtrees and the degree moments. Here and in the following, we write $V(T)$ and $E(T)$ for the vertex set and edge set of a tree $T$, $|T|$ is the number of vertices of $T$, and $d(v)$ denotes the degree of a vertex $v$; whenever we speak about the degree of a vertex, we always mean the degree in the underlying tree $T$, not a subtree.

\begin{theo}\label{thm:main}
Let $T_1,T_2,\ldots$ be a sequence of trees such that $|T_n| \to \infty$ as $n \to \infty$. For every $k \geq 4$, the following four statements are equivalent:
\begin{enumerate}
\item[(M1)] $\displaystyle \lim_{n \to \infty} p_1^{(k)}(T_n) = 0$,
\item[(M2)] $\displaystyle \lim_{n \to \infty} \frac1{|T_n|} Z_k(T_n) = \infty$,
\item[(M3)] $\displaystyle \lim_{n \to \infty} \frac1{|T_n|} \sum_{v \in V(T_n)} d(v)^{k-1} = \infty$,
\item[(M4)] $\displaystyle \lim_{n \to \infty} p_2^{(k)}(T_n) = 1$.
\end{enumerate}
\end{theo}
Informally, statement (M2) says that $T_n$ contains more than linearly many $k$-vertex subtrees. (M3) states that the $(k-1)$-th degree moment tends to infinity. The implication (M4) $\Rightarrow$ (M1) is trivial, so our main task will be to prove the implications (M1) $\Rightarrow$ (M2) $\Leftrightarrow$ (M3) $\Rightarrow$ (M4).

Shortly after a first version of this paper was published online, the equivalence of (M1) and (M4) was shown independently by Bubeck, Edwards, Mania and Supko \cite{bubeck2016paths}, who also provided an explicit (nonlinear) inequality between $p_1^{(k)}(T)$ and $p_2^{(k)}(T)$ that implies the equivalence.

\section{Proof of the main theorem}\label{sec:main}

Theorem~\ref{thm:main} will follow from a sequence of lemmas.  As a first step, we estimate the total number of $k$-vertex subtrees. 

\begin{lemma}\label{lem:simple_bound}
Let $k$ be a positive integer. The total number of $k$-vertex subtrees of any tree $T$ can be bounded above as follows:
$$Z_k(T) \leq (k-1)!\sum_{v \in V(T)} d(v)^{k-1}.$$
\end{lemma}

\begin{proof}
For every vertex $v$ of $T$, we count the number of $k$-vertex subtrees with the property that $v$ is contained, and that it has maximum degree (in $T$, not the subtree!) among all vertices of the subtree. Every such subtree can be constructed by repeatedly adding a leaf, starting with the single vertex $v$. At the $j$-th such step, there are at most $j$ vertices to attach a leaf to, and at most $d(v)$ choices for the new leaf (since $v$ was assumed to have maximum degree). Therefore, there are at most $(k-1)! \cdot d(v)^{k-1}$ possible subtrees of this kind for any fixed vertex $v$. Summing over all $v$, we obtain the desired result. Clearly every subtree is counted at least once in the sum---possibly even several times, but since we are only interested in an upper bound, this is immaterial.
\end{proof}

\begin{lemma}\label{lem:stars}
For every integer $k \geq 3$, the total number of $k$-vertex stars contained in a tree $T$ is
$$c(S_k,T) = \sum_{v \in V(T)} \binom{d(v)}{k-1}.$$
\end{lemma}

\begin{proof}
The number of $k$-vertex stars contained in $T$ whose center is $v$ is given by $\binom{d(v)}{k-1}$, the number of ways to choose $k-1$ of its neighbors. The desired statement follows immediately.
\end{proof}

Note that $p_1^{(k)}(T_n) = 0$ if the diameter of $T_n$ is at most $k-2$ (in this case, there are certainly no induced $k$-vertex paths), so this would provide us with a simple construction for which condition (M1) holds. We will treat this case separately and show that it implies (M2):

\begin{lemma}\label{lem:bounded_diam}
Fix an integer $k \geq 3$, and let $T_1,T_2,\ldots$ be a sequence of trees whose diameter is bounded above by some fixed constant $D$. If $|T_n| \to \infty$ as $n \to \infty$, then (M2) holds, i.e.
$$\lim_{n \to \infty} \frac{1}{|T_n|} Z_k(T_n) = \infty.$$
\end{lemma}

\begin{proof}
We prove the slightly stronger statement that
$$\lim_{n \to \infty} \frac{1}{|T_n|} c(S_k,T_n) = \infty,$$
i.e.~the number of induced $k$-vertex stars grows faster than linearly. To this end, it will be useful to consider all trees as rooted (at an arbitrary vertex). Clearly, if the diameter is bounded by $D$, the height of any rooted version is also bounded by $D$. We prove the following by induction on $D$, from which the statement of the lemma follows immediately:

{\bf Claim.} For every positive integer $D$, there exist positive constants $\alpha_{D}$, $\beta_{D}$ with $\beta_{D} > 1$ and a positive integer $N_D$ depending only on $D$ and $k$ such that
$$c(S_k,T) \geq \alpha_{D}\max(|T|-N_D,0)^{\beta_D}$$
for any rooted tree $T$ whose height is at most $D$.

 First note that the claim is trivial for $D=1$: there is only one possible rooted tree in this case, namely a star. Thus we have
$$c(S_k,T) = \binom{|T|-1}{k-1} \geq \Big( \frac{|T|}{k} \Big)^{k-1}$$
in this case as soon as $|T| \geq k$, which gives us the desired inequality with $\beta_1 = k-1 > 1$, $\alpha_1 = k^{-(k-1)}$ and $N_1 = k$.

Now we turn to the induction step. Let $r$ be the root degree, and let $T_1,T_2,\ldots,T_r$ be the root branches, each endowed with the natural root (the neighbor of $T$'s root). The number of copies of $S_k$ in $T$ for which the root is the centre is given by $\binom{r}{k-1}$, so
$$c(S_k,T) \geq \binom{r}{k-1} + \sum_{j=1}^r c(S_k,T_j).$$
Each of the branches has height at most $D-1$, so we can apply the induction hypothesis to them. In addition, we note that $f(x) = \alpha_{D-1}\max(x-N_{D-1},0)^{\beta_{D-1}}$ is a convex function, so Jensen's inequality gives us
$$c(S_k,T) \geq \binom{r}{k-1} + r \alpha_{D-1} \max \Big( \frac{|T|-1}{r}-N_{D-1},0\Big)^{\beta_{D-1}}.$$
If $r \geq |T|^{2/3}$ and $|T| \geq (k-1)^{3/2}$, then the first term is 
$$\binom{r}{k-1} \geq \binom{|T|^{2/3}}{k-1} \geq \Big( \frac{|T|^{2/3}}{k-1} \Big)^{k-1}.$$
If, on the other hand, $r < |T|^{2/3}$ and $|T| \geq (N_{D-1}+2)^3$, then the second term is
\begin{align*}
r \alpha_{D-1} \max \Big( \frac{|T|-1}{r}-N_{D-1},0\Big)^{\beta_{D-1}} &\geq r \alpha_{D-1} \Big( \frac{|T|}{r}-N_{D-1}-1 \Big)^{\beta_{D-1}} \\ 
&\geq r \alpha_{D-1} \Big( \frac{|T|}{r(N_{D-1}+2)} \Big)^{\beta_{D-1}} \\
&= \frac{\alpha_{D-1}}{(N_{D-1}+2)^{\beta_{D-1}}} \cdot r^{1-\beta_{D-1}}|T|^{\beta_{D-1}} \\
&\geq \frac{\alpha_{D-1}}{(N_{D-1}+2)^{\beta_{D-1}}} \cdot |T|^{(\beta_{D-1}+2)/3}.
\end{align*}
Thus we obtain the desired inequality with
\begin{align*}
\alpha_{D} &= \min \Big(\frac{1}{(k-1)^{k-1}}, \frac{\alpha_{D-1}}{(N_{D-1}+2)^{\beta_{D-1}}} \Big), \\
\beta_{D} &= \min \Big(\frac{2(k-1)}{3}, \frac{\beta_{D-1}+2}{3} \Big), \\
N_D &= \max \Big( (k-1)^{3/2}, (N_{D-1}+2)^3 \Big).
\end{align*}
Since $k \geq 3$ and we were assuming $\beta_{D-1} > 1$, we also have $\beta_{D} > 1$. This completes the induction and thus the proof of the lemma.
\end{proof}

Lemma~\ref{lem:bounded_diam} shows that (M2) always holds for sequences of trees with bounded diameter, even without the assumption (M1). On the other hand, if the diameter is sufficiently large, then it turns out that there must always be at least linearly many paths of length $k$. In fact, we have the following simple lemma:

\begin{lemma}\label{lem:many_paths}
Let $k$ be a positive integer. If the diameter of a tree $T$ is at least $2k-2$, then $c(P_k,T) \geq |T|/2$. 
\end{lemma}

\begin{proof}
Since the diameter is assumed to be at least $2k-2$, the radius must be at least $k-1$. Therefore, for every vertex $v$ of $T$, there is a $k$-vertex path in $T$ starting at $v$. Since every path has only two ends, no path is counted more than twice in this argument, thus there must be at least $|T|/2$ $k$-vertex paths occurring in $T$.
\end{proof}

\begin{cor}\label{cor_m1m2}
For every integer $k \geq 4$, the implication (M1) $\Rightarrow$ (M2) holds.
\end{cor}

\begin{proof}
Consider a sequence $T_1,T_2,\ldots$ of trees with $|T_n| \to \infty$ for which (M1) holds. For the subsequence consisting of trees whose diameter is at most $2k-3$, (M2) follows from Lemma~\ref{lem:bounded_diam}, regardless of whether (M1) is true or not. For the remaining subsequence, we can simply combine  Lemma~\ref{lem:many_paths} with the assumption (M1).
\end{proof}

As a next step, we show the equivalence of (M2) and (M3), which is quite straightforward:

\begin{lemma}\label{lem:m2_m3}
For every integer $k \geq 3$, the two statements (M2) and (M3) are equivalent.
\end{lemma}

\begin{proof}
Condition (M2), combined with Lemma~\ref{lem:simple_bound}, implies that
$$\lim_{n \to \infty} \frac{1}{|T_n|} \sum_{v \in V(T_n)} d(v)^{k-1}  = \infty,$$
which is exactly (M3). On the other hand, since $\binom{d}{k-1} \geq \big( \frac{d}{k-1} \big)^{k-1}$ for $d \geq k-1$, Lemma~\ref{lem:stars} gives
\begin{equation}\label{eq:m2_and_m3}
c(S_k,T_n) \geq (k-1)^{-(k-1)} \sum_{v \in V(T_n)} d(v)^{k-1}  - |T_n|,
\end{equation}
where the final term stems from vertices whose degree is less than $k-1$. Therefore, if (M3) holds, then we also have
$$
\lim_{n \to \infty} \frac{c(S_k,T_n)}{|T_n|} = \infty,
$$
which is (M2).
\end{proof}

Now we would like to bound the number of non-star $k$-vertex subtrees from above to obtain the implication (M2) $\Rightarrow$ (M4). To this end, we first introduce the notion of edge weights:

Define the \emph{weight} of an edge $e = vu$ as
$$\omega(e) = \max \Big( \frac{d(u)}{d(v)}, \frac{d(v)}{d(u)} \Big).$$
In words: take the degrees of the two endpoints of $e$ and divide the higher degree by the lower degree. For some real number $a > 1$, call a subtree $S$ of a tree $T$ an $a$-\emph{unbalanced} subtree if it contains at least one edge that is not a pendant edge (incident to a leaf) of $S$ and that has a weight of at least $a$ in $T$. Denote the total number of $a$-unbalanced $k$-vertex subtrees of $T$ by $Z_k(T,a)$. The following lemma is in some sense a refinement of Lemma~\ref{lem:simple_bound}. 

\begin{lemma}\label{lem:unbalanced}
For every integer $k \geq 4$, every real number $a > 1$, and every tree $T$, we have
$$Z_k(T,a) \leq \frac{(k-1)!}{a} \sum_{v \in V(T)} d(v)^{k-1}.$$
\end{lemma}

\begin{proof}
We can follow the proof of Lemma~\ref{lem:simple_bound}. The only change in the argument is that at least one vertex of degree at most $d(v)/a$ has to be added to the subtree at some point so as to include an edge of weight at least $a$. Since we also require the presence of such an edge that is not a pendant edge of the subtree, at some stage a neighbor of a vertex of degree at most $d(v)/a$ has to be added to the subtree as well, for which there are only at most $d(v)/a$ possibilities. This gives us the same inequality as in Lemma~\ref{lem:simple_bound}, but with an extra factor $a$ in the denominator.
\end{proof}

It remains to bound the number of $k$-vertex subtrees that are neither stars nor $a$-unbalanced; we denote this number by $\overline{Z}_k(T,a)$. Our next lemma provides a suitable bound:

\begin{lemma}\label{lem:balanced}
For every integer $k \geq 4$, every real number $a > 1$, and every tree $T$, we have
$$\overline{Z}_k(T,a) \leq 2(k-1)! a^{(k-2)^2} \sum_{v \in V(T)} d(v)^{k-2}.$$
\end{lemma}

\begin{proof}
Consider any edge $e$ whose weight is at most $a$. It is not difficult to see that there exists some nonnegative integer $\ell$ such that the degrees of both its endpoints lie in the interval $[a^\ell,a^{\ell+2})$: simply take $\ell$ in such a way that the smaller degree of the two lies in $[a^\ell,a^{\ell+1})$. Now consider any subtree $S$ that is not $a$-unbalanced and contains $e$ as a non-pendant edge (it automatically follows that $S$ is not a star). Every internal vertex $v$ of $S$ can be reached from $e$ by a path of non-pendant edges whose length is at most $k-4$. Since $S$ was assumed not to be $a$-unbalanced, none of these edges can have a weight greater than $a$, so the degree of $v$ in $T$ is at most $a^{\ell+2} \cdot a^{k-4} = a^{\ell+k-2}$.

Now we count all subtrees $S$ that are not $a$-unbalanced and contain $e$ as a non-pendant edge. Every such subtree can be obtained by repeatedly adding leaves, starting from $e$. This is done $k-2$ times. At the $j$-th step, we have a choice of $j+1$ vertices to attach a leaf to, and at most $a^{\ell+k-2}$ possible choices for the leaf by the observation on degrees of internal vertices in $S$. It follows that there are no more than
$$(k-1)! \cdot a^{(k-2)(\ell+k-2)}$$
such subtrees.

The number of edges whose ends both have degrees in $[a^\ell,a^{\ell+2})$ is less than the number of vertices whose degrees lie in this interval, since the edges induce a forest on the set of these vertices. Therefore, we obtain the following  upper bound for the number of $k$-vertex subtrees that are neither stars nor $a$-unbalanced (note that every non-star has at least one non-pendant edge):
\begin{align*}
\overline{Z}_k(T,a) &\leq \sum_{\ell \geq 0} \sum_{\substack{e = vw \in E(T) \\ d(v),d(w) \in [a^{\ell},a^{\ell+2})}} (k-1)! \cdot a^{(k-2)(\ell+k-2)} \\
&\leq \sum_{\ell \geq 0} \sum_{\substack{v \in V(T) \\ d(v) \in [a^{\ell},a^{\ell+2})}} (k-1)! \cdot a^{(k-2)(\ell+k-2)} \\
&\leq  \sum_{\ell \geq 0} \sum_{\substack{v \in V(T) \\ d(v) \in [a^{\ell},a^{\ell+2})}} (k-1)! \cdot d(v)^{k-2} \cdot a^{(k-2)^2} \\
&\leq 2(k-1)! a^{(k-2)^2} \sum_{v \in V(T)} d(v)^{k-2}.
\end{align*}
The last inequality holds since every vertex is counted at most twice in the double sum.
\end{proof}

Now we put everything together to obtain the desired implication (M2) $\Rightarrow$ (M4), completing the proof of Theorem~\ref{thm:main}. Let us formulate this explicitly:

\begin{cor}\label{cor:m2m4}
For every integer $k \geq 4$, the implication (M2) $\Rightarrow$ (M4) holds.
\end{cor}

\begin{proof}
Assume that condition (M2) is satisfied. Combining it with inequality~\eqref{eq:m2_and_m3} from the proof of Lemma~\ref{lem:m2_m3}, we see that
$$\frac{1}{c(S_k,T_n)} \sum_{v \in V(T_n)} d(v)^{k-1}$$
is bounded above by a positive constant (for sufficiently large $n$).

We combine Lemma~\ref{lem:unbalanced} and~Lemma \ref{lem:balanced} to find that the total number of $k$-vertex subtrees of a tree $T$ that are not stars can be bounded by
$$Z_k(T,a) + \overline{Z}_k(T,a) \leq \frac{(k-1)!}{a} \sum_{v \in V(T)} d(v)^{k-1} + 2(k-1)! a^{(k-2)^2} \sum_{v \in V(T)} d(v)^{k-2}.$$
H\"older's inequality gives us
$$\sum_{v \in V(T)} d(v)^{k-2} \leq |T|^{1/(k-1)} \bigg( \sum_{v \in V(T)} d(v)^{k-1} \bigg)^{(k-2)/(k-1)},$$
so putting everything together, we obtain
\begin{align*}
Z_k(T_n) - c(S_k,T_n) &= Z_k(T_n,a) + \overline{Z}_k(T_n,a) \\
&= O \Bigg( a^{-1} \sum_{v \in V(T_n)} d(v)^{k-1} + a^{(k-2)^2} \sum_{v \in V(T_n)} d(v)^{k-2} \Bigg) \\
&= O \Bigg( a^{-1} \sum_{v \in V(T_n)} d(v)^{k-1} + a^{(k-2)^2} |T_n|^{1/(k-1)} \bigg( \sum_{v \in V(T_n)} d(v)^{k-1} \bigg)^{(k-2)/(k-1)} \Bigg) \\
&= O \Big( a^{-1} c(S_k,T_n) + a^{(k-2)^2} |T_n|^{1/(k-1)} c(S_k,T_n)^{(k-2)/(k-1)} \Big).
\end{align*}
The $O$-constant depends on $k$ and the specific sequence of trees, but notably not on $a$, which we can still choose freely. Taking
$$a = \Big( \frac{c(S_k,T_n)}{|T_n|} \Big)^{\frac{1}{(k-1)((k-2)^2+1)}},$$
which is greater than $1$ for sufficiently large $n$ in view of condition (M2), the two terms in the estimate balance, and we end up with
$$Z_k(T_n) - c(S_k,T_n) = O \Big( \Big( \frac{|T_n|}{c(S_k,T_n)} \Big)^{\frac{1}{(k-1)((k-2)^2+1)}} c(S_k,T_n) \Big),$$
so that~(M2) now implies
$$\lim_{n \to \infty} \frac{c(S_k,T_n)}{Z_k(T_n)} = 1,$$
which is exactly (M4).
\end{proof}

As we have now shown the implications (M1) $\Rightarrow$ (M2) (Corollary~\ref{cor_m1m2}), (M2) $\Leftrightarrow$ (M3) (Lemma~\ref{lem:m2_m3}) and (M2) $\Rightarrow$ (M4) (Corollary~\ref{cor:m2m4}) and the implication (M4) $\Rightarrow$ (M1) is trivial, this also completes the proof of Theorem~\ref{thm:main}.

Our ideas can also be used to re-prove a result of Bubeck and Linial \cite[Theorem 2]{bubeck2015local}, even with a slightly improved constant: namely, they showed that
$$\liminf_{n \to \infty} \Big( p_1^{(k)}(T_n) +  p_2^{(k)}(T_n) \Big) \geq \frac{1}{2k^{2k}N_k}$$
for any sequence $T_1,T_2,\ldots$ of trees with $|T_n| \to \infty$, where $N_k$ is the number of nonisomorphic trees with $k$ vertices.

Making use of the arguments used to prove Theorem~\ref{thm:main}, we obtain the following:

\begin{cor}\label{cor:lower_bound_mix}
For every sequence $T_1,T_2,\ldots$ of trees with $|T_n| \to \infty$, we have
$$\liminf_{n \to \infty} p_1^{(k)}(T_n) +  p_2^{(k)}(T_n) \geq \frac{1}{2(k-1)^{k-1}(k-1)!}.$$
\end{cor}

\begin{proof}
Lemma~\ref{lem:simple_bound} gives us
$$Z_k(T_n) \leq (k-1)! \sum_{v \in V(T_n)} d(v)^{k-1}.$$
Combining this inequality with~\eqref{eq:m2_and_m3} and Lemma~\ref{lem:many_paths} (we may assume that the diameter is not bounded in view of Lemma~\ref{lem:bounded_diam}) yields
$$Z_k(T_n) \leq (k-1)! (k-1)^{k-1} (c(S_k,T_n) + |T_n|) \leq (k-1)! (k-1)^{k-1} (c(S_k,T_n) + 2c(P_k,T_n)).$$
Therefore,
\begin{align*}
p_1^{(k)}(T_n) +  p_2^{(k)}(T_n) &= \frac{c(S_k,T_n) + c(P_k,T_n)}{Z_k(T_n)} \\
&\geq \frac{c(S_k,T_n) + c(P_k,T_n)}{(k-1)! (k-1)^{k-1} (c(S_k,T_n) + 2c(P_k,T_n))},
\end{align*}
and the desired result follows immediately.

\end{proof}

With more careful estimates, it is certainly possible to improve further on the lower bound in Corollary~\ref{cor:lower_bound_mix}.

\section{Subtrees of different sizes}

So far, we were only comparing subtrees of the same fixed size $k$. However, it is natural to assume that $\lim_{n \to \infty} p_1^{(k)}(T_n) = 0$ for some $k$ (in words: the proportion of paths among $k$-vertex subtrees goes to $0$) should also imply $\lim_{n \to \infty} p_2^{(\ell)}(T_n) = 1$ (the proportion of stars among $\ell$-vertex subtrees goes to $1$) for some $\ell$ that is not necessarily equal to $k$. Indeed this is true if $k \leq \ell$: since we trivially have
$$\sum_{v \in V(T)} d(v)^{k-1} \leq \sum_{v \in V(T)} d(v)^{\ell-1}$$
in this case, condition (M3) is satisfied for $\ell$ if it is satisfied for $k$. Therefore, we immediately obtain a slight extension of Theorem~\ref{thm:main}:

\begin{theo}
Let $T_1,T_2,\ldots$ be a sequence of trees such that $|T_n| \to \infty$ as $n \to \infty$. Let $k,\ell$ be integers such that $\ell \geq k \geq 4$, and assume that one of the following equivalent statements holds:
\begin{enumerate}
\item[(M1)${}_k$] $\displaystyle \lim_{n \to \infty} p_1^{(k)}(T_n) = 0$,
\item[(M2)${}_k$] $\displaystyle \lim_{n \to \infty} \frac1{|T_n|} Z_k(T_n) = \infty$,
\item[(M3)${}_k$] $\displaystyle \lim_{n \to \infty} \frac1{|T_n|} \sum_{v \in V(T_n)} d(v)^{k-1} = \infty$,
\item[(M4)${}_k$] $\displaystyle \lim_{n \to \infty} p_2^{(k)}(T_n) = 1$.
\end{enumerate}
In this case, the following statements hold as well:
\begin{enumerate}
\item[(M1)${}_\ell$] $\displaystyle \lim_{n \to \infty} p_1^{(\ell)}(T_n) = 0$,
\item[(M2)${}_\ell$] $\displaystyle \lim_{n \to \infty} \frac1{|T_n|} Z_{\ell}(T_n) = \infty$,
\item[(M3)${}_\ell$] $\displaystyle \lim_{n \to \infty} \frac1{|T_n|} \sum_{v \in V(T_n)} d(v)^{\ell-1} = \infty$,
\item[(M4)${}_\ell$] $\displaystyle \lim_{n \to \infty} p_2^{(\ell)}(T_n) = 1$.
\end{enumerate}
\end{theo}

In heuristic terms: if most $k$-vertex subtrees are stars, then this is also the case for $\ell$-vertex subtrees, provided $\ell \geq k$. On the other hand, if only very few of the $k$-vertex subtrees are paths, then the same applies to $\ell$-vertex subtrees for every $\ell \geq k$. It is noteworthy, however, that the converse is not true, and counterexamples are very easy to construct.

Consider for instance a family of extended stars constructed as follows (Figure~\ref{fig:ext_star}): $T_n$ has $n$ vertices, of which the central vertex has degree (approximately) $n^{2/(2k-1)}$ for some $k \geq 4$, while all other vertices have degree $1$ or $2$. The actual lengths of the paths around the central vertex are irrelevant. It is easy to see in this example that (M3)${}_k$ is not satisfied, and that in fact $\lim_{n \to \infty} p_2^{(k)}(T_n) = 0$, while on the other hand (M3)${}_{k+1}$ is satisfied, so that $\lim_{n \to \infty} p_2^{(k+1)}(T_n) = 1$.

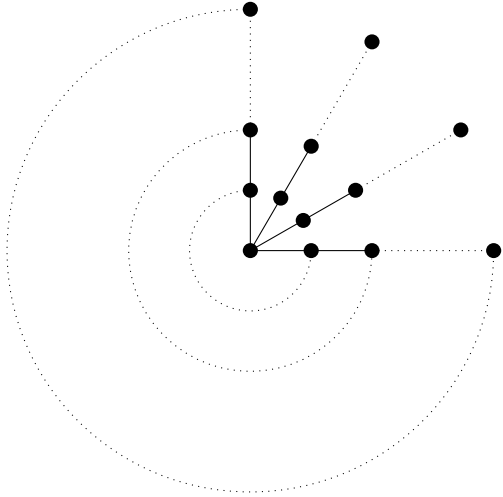
\begin{figure}[htbp]
\begin{center}
\begin{tikzpicture}[scale=0.8]

        \node[fill=black,circle,inner sep=2pt]  at (0,0) {};
        \node[fill=black,circle,inner sep=2pt]  at (0,1) {};
        \node[fill=black,circle,inner sep=2pt]  at (0,2) {};
        \node[fill=black,circle,inner sep=2pt]  at (0,4) {};
        \node[fill=black,circle,inner sep=2pt]  at (0.5,0.87) {};
        \node[fill=black,circle,inner sep=2pt]  at (1,1.73) {};
        \node[fill=black,circle,inner sep=2pt]  at (2,3.46) {};
        \node[fill=black,circle,inner sep=2pt]  at (0.87,0.5) {};
        \node[fill=black,circle,inner sep=2pt]  at (1.73,1) {};
        \node[fill=black,circle,inner sep=2pt]  at (3.46,2) {};
        \node[fill=black,circle,inner sep=2pt]  at (1,0) {};
        \node[fill=black,circle,inner sep=2pt]  at (2,0) {};
        \node[fill=black,circle,inner sep=2pt]  at (4,0) {};

	\draw (0,0)--(0,2);
	\draw [dotted] (0,2)--(0,4);
	\draw (0,0)--(1,1.73);
	\draw [dotted] (1,1.73)--(2,3.46);
	\draw (0,0)--(1.73,1);
	\draw [dotted] (1.73,1)--(3.46,2);
	\draw (0,0)--(2,0);
	\draw [dotted] (2,0)--(4,0);

	\draw [dotted] (0,1) arc (90:360:1) ;
	\draw [dotted] (0,2) arc (90:360:2) ;
	\draw [dotted] (0,4) arc (90:360:4) ;

\end{tikzpicture}
\end{center}
\caption{An extended star.}\label{fig:ext_star}
\end{figure}


\end{document}